\documentclass{amsart}
\usepackage{a4wide,url}

\newcommand{\asdim}{\mathrm{asdim}}
\newcommand{\E}{\mathcal E}
\newcommand{\F}{\mathcal F}
\newcommand{\U}{\mathcal U}
\newcommand{\V}{\mathcal V}

\newcommand{\w}{\omega}
\newcommand{\IZ}{\mathbb Z}
\newcommand{\IR}{\mathbb R}

\newcommand{\dist}{\mathrm{dist}}

\newtheorem{theorem}{Theorem}

\newtheorem{lemma}{Lemma}
\newtheorem{corollary}{Corollary}

\title{On the asymptotic dimension of  products of coarse spaces}
\author{Iryna Banakh and Taras Banakh}
\dedicatory{Dedicated to Sergey Antonyan on the occasion of his 65th birthday}
\address{I.Banakh: Ya. Pidstryhach Institute for Applied Problems of Mechanics and Mathematics of NASU, Naukova 3b, Lviv}
\address{T.Banakh: Ivan Franko National University of Lviv (Ukraine) and Jan Kochanowski University in Kielce (Poland)}
\email{ibanakh@yahoo.com, t.o.banakh@gmail.com}
\subjclass{52C17; 55M10}
\keywords{Coarse space, product, asymptotic dimension, G-space}
\begin{document}

\begin{abstract}We prove that for any coarse spaces $X_1,\dots,X_n$ of asymptotic dimension $\ge 1$, the product $X=X_1\times\dots\times X_n$ has asymptotic dimension $\ge n$. Another result states that a finitary coarse space $Z$ has $\asdim (Z)\ge n$ if $Z$ admits an almost free action of the group $\IZ^n$. We deduce these inequalities from the following combinatorial result (that generalizes  the Hex Theorem of Gale): for any cover $\F$ of a discrete box $K=k_1\times \dots \times k_n$, either some set $F\in\F$ contains a chain connecting two opposite faces of $K$ or there exists a set $B\subseteq K$ of diameter $\le 1$ that intersects at least $n+1$ elements of the cover $\F$.
\end{abstract} 

\maketitle

\section{Introduction}

The asymptotic dimension of a metric (more generally, coarse) space is an important number characteristic of the space, introduced by Gromov \cite{Gromov}. The asymptotic dimension is a coarse counterpart of the Lebesgue covering dimension of topological spaces. By a result of Holszty\'nski \cite{Hol} and Lifanov \cite{Lif} (see also \cite[1.8.K]{Eng}), for any compact Hausdorff spaces $X_1,\dots,X_n$ of dimension $\dim X_i\ge 1$ the product $X=X_1\times \dots \times X_n$ has covering dimension $\dim X\ge n$. In this paper we prove a counterpart of this result in the category of coarse spaces, answering a question of Protasov. Namely, we prove that for coarse spaces $X_1,\dots,X_n$ of asymptotic dimension $\asdim(X_i)\ge 1$, the product $X=X_1\times \dots \times X_n$ has asymptotic dimension $\asdim(X)\ge n$.
Also we prove the inequality $\asdim X\ge n$ for any finitary coarse space $X$ admitting an almost free action of the group $\IZ^n$. This implies that for any infinite set $X$ endowed with the finitary coarse structure $\E$ generated by the group $S_X$ of bijections  of $X$, the coarse space $(X,\E)$ has infinite asymptotic dimension. This answers another problem of Protasov.

We deduce these results from a combinatorial property of discrete boxes, which generalizes the Hex Theorem of Gale \cite{Gale}, and is a corollary of the Lebesgue Covering Lemma 1.8.20 in \cite{Eng}. Namely, we prove that for any cover $\F$ of the product $K=k_1\times\dots\times k_n$ of nonzero finite ordinals, either some set $F\in\F$ connects two opposite faces of the box $K$ or some set $B\subseteq K$ of diameter $\le 1$ meets more than $n$ elements of the cover $\F$. 

\section{Preliminaries}

In this section we recall the necessary definitions from Asymptology.
\smallskip

\subsection{Coarse spaces} An {\em entourage} on a set $X$ is any subset $E$ of the square $X\times X$ such that $\Delta_X\subseteq E=E^{-1}$ where $\Delta_X=\{(x,x):x\in X\}$ and $E^{-1}=\{(y,x):(x,y)\in E\}$. For two entourages $E,F$ on $X$ we put 
$$E\circ F:=\{(x,z):\exists y\;\;(x,y)\in E,\;(y,z)\in F\}.$$For an entourage $E$ on $X$ and a point $x\in X$, the set $$E[x]=\{y\in X:(x,y)\in E\}$$ is called the $E$-ball around $x$. 

A {\em coarse structure} on a set $X$ is a family $\E$ of entourages on $X$, satisfying the following axioms:
\begin{itemize}
\item for any $E,F\in\E$ we have $E\circ F\in \E$;
\item an entourage $F$ on $X$ belongs to $\E$ if there exists an entourage $E\in\E$ such that $F\subseteq E$.
\end{itemize}
A subfamily $\mathcal B\subseteq\E$ is called a {\em base} of the coarse structure $\E$ if for each $E\in\E$ there exists $B\in\mathcal B$ such that $E\subseteq B$.

A {\em coarse space} is a pair $(X,\E)$ consisting of a set $X$ and a coarse structure $\E$ on $X$.

Each metric $d$ on $X$ generates the coarse structure
$$\E=\bigcup_{r\in[0,\infty)}\big\{E\subseteq X\times X:\Delta_X\subseteq E=E^{-1}\subseteq\{(x,y)\in X\times X:d(x,y)\le r\}\big\}$$
called the {\em canonical coarse structure} of the metric space $(X,d)$. This coarse structure is generated by the base
$$\mathcal B=\big\{\{(x,y)\in X\times X:d(x,y)\le r\}:r\ge 0\big\}.$$

For coarse spaces $(X_1,\E_1),\dots(X_n,\E_n)$, their product is the Cartesian product $X=X_1\times\dots\times X_n$ endowed with the coarse structure $\E$ generated by the base 
$$\mathcal B=\Big\{\big\{\big((x_i)_{i=1}^n,(y_i)_{i=1}^n\big)\in X\times X:\big((x_i,y_i)\big)_{i=1}^n\in\prod_{i=1}^nE_i\big\}:(E_i)_{i=1}^n\in\prod_{i=1}^n\E_i\Big\}.$$  
If for every $i\in\{1,\dots,n\}$ the coarse structure $\E_i$ is generated by a metric $d_i$, then the coarse structure $\E$ on $X=X_1\times \dots\times X_n$ is generated by the metric $$d:X\times X\to[0,\infty),\;\;d\big((x_i)_{i=1}^n,(y_i)_{i=1}^n\big)=\max_{1\le i\le n}d_i(x_i,y_i).$$ 
More information on coarse spaces can be found in \cite{PB}, \cite{PZ} and  \cite{Roe}.
 \medskip
 
\subsection{Finitary coarse spaces} A coarse space $(X,\E)$ is called
\begin{itemize}
\item {\em locally finite} if for any $E\in\E$ and $x\in X$ the $E$-ball $E[x]$ is finite;
\item {\em finitary} if $\sup_{x\in X}|E[x]|<\infty$ for any entourage $E\in\E$. 
\end{itemize}
Here $|E[x]|$ stands for the cardinality of the $E$-ball $E[x]$. It is clear that each finitary coarse space is locally finite.
 In some papers {\em finitary} coarse structures are called {\em uniformly locally finite}. 

Let us recall that a {\em $G$-space} is a pair $(X,G)$ consisting of a set $X$ and a subgroup $G$ of the group $S_X$ of all bijections of $X$. Each $G$-space $(X,G)$ carries a canonical finitary coarse structure $\E_G$, generated by the base consisting of the enourages
$$\{(x,y)\in X\times X:y\in \{x\}\cup Fx\cup F^{-1}x\}$$where $F$ runs over finite subsets of $G$. By Theorem 1 in \cite{Prot} (see also \cite{PP} and \cite{P19}), any  finitary coarse structure on a set $X$ is equal to the coarse structure $\E_G$ for a suitable subgroup $G\subseteq S_X$. This fundamental result of Protasov implies that for any set $X$, the coarse structure $\E_{S_X}$ is the largest finitary coarse structure on $X$. 

On the other hand, the largest locally finite coarse structure on $X$ consists of all entourages $E$ on $X$ such that for every $x\in X$ the sets $E[x]$ and $E^{-1}[x]$ are finite. For more information on the largest locally finite coarse structure, see \cite[\S6]{BP}.

We shall say that a coarse space $(X,\E)$ admits an {\em almost free action of a group $H$} if there exists an isomorphic copy $G\subseteq S_X$ of the group $H$ such that
$\E_G\subseteq \E$ and for any finite set $F\subseteq G\setminus\{e\}$ there exists a point $x\in X$ such that $\{g\in G:gx=x\}$ is disjoint with $F$. Here $e$ stands for the neutral element of the group $G$.

\subsection{Asymptotic dimension} Let $X$ be a coarse space and $\E$ be its coarse structure. A cover $\U$ of $X$ is called {\em uniformly bounded} if there exists an entourage $E\in\E$ such that $U\times U\subseteq E$ for every $U\in\U$.

A coarse space $X$ is defined to have {\em finite asymptotic dimension} if there exists a non-negative integer $n$ such that for any entourage $E\in\E$ there exists a uniformly bounded cover $\U$ of $X$ such that for every $x\in X$ the family $\U_x=\{U\in\U:E[x]\cap U\ne\emptyset\}$ has cardinality $|\U_x|\le n+1$. The smallest number $n$ with this property is called the {\em asymptotic dimension} of $X$ and is denoted by $\asdim(X)$. If $X$ fails to have finite asymptotic dimension, then we put $\asdim(X)=\infty$ and say that $X$ has {\em infinite asymptotic dimension}. Here we assume that $\infty \ge n$ for any $n\in\w$.

More information on asymptotic dimension can be found in the surveys \cite{BD}, \cite{BD2}, \cite{Grave}. These surveys mention many results establishing upper bounds for the asymptotic dimension of coarse spaces. However, we could not find general results giving lower bounds for the asymptotic dimension of products. In this note we provide such lower bounds.
\medskip

\section{Main results}

The main results of this paper are the following theorems that answer two questions of Protasov (asked in an e-mail correspondence).

\begin{theorem}\label{t:main}
Let $X_1,\dots,X_n$ be coarse spaces of asymptotic dimension $\ge 1$. Then the product $X=X_1\times\dots\times X_n$ has $\asdim(X)\ge n$.
\end{theorem}

\begin{theorem}\label{t:main2} If for some positive integer $n$ a finitary coarse space $(X,\E)$ admits an almost free action of the group $\IZ^n$, then $\asdim(X,\E)\ge n$.
\end{theorem}

\begin{corollary}\label{cor} For any infinite set $X$ and the group $G=S_X$ 
the coarse space $(X,\E_G)$ has infinite asymptotic dimension.
\end{corollary}

Corollary~\ref{cor} should be compared with Theorem 1 \cite{P-DZ} saying that for any infinite set $X$ and the largest locally finite coarse structure $\mathcal F$ on $X$, the coarse space $(X,\mathcal F)$ has $\asdim (X,\mathcal F)=1$. 

Theorems~\ref{t:main} and \ref{t:main2} will be proved in Section~\ref{s:pf1}, \ref{s:pf2} after some preliminary work made in Sections~\ref{s:d} and \ref{s:l}.

\section{A dimension dichotomy for discrete boxes}\label{s:d}

In this section we establish a dimension dichotomy for discrete boxes, which can be considered as a local generalization of the Hex Theorem of Gale \cite{Gale}. 

By a {\em discrete box of dimension $n$} we understand the product $K=k_1\times\dots\times k_n$ of nonzero finite ordinals $k_1,\dots,k_n$. The discrete box $K$ is endowed with the metric $d$ defined by 
$$d(x,y)=\max_{1\le i\le n}|x(i)-y(i)|.$$
Here we identify each finite ordinal $k$ with the subset $\{0,\dots,k-1\}$ of the smallest infinite ordinal $\w=\{0,1,2,\dots\}$. Elements of the product $K=k_1\times \dots\times k_n$ are functions $x:\{1,\dots,n\}\to\w$ such that $x(i)\in k_i=\{0,\dots,k_i-1\}$  for all $i\in\{1,\dots,n\}$. 

A subset $C$ of a box $K$ is defined to be {\em chain-connected} if for any points $x,y\in C$ there exists a sequence of points $x=x_0,\dots,x_m=y$ such that $d(x_{i-1},x_i)\le 1$ for any $i\in\{1,\dots,m\}$.
 
We say that a set $F\subseteq K$ {\em connects two opposite faces} of a box $K=k_1\times\dots\times k_n$ if there exist $i\in\{1,\dots,n\}$ and a sequence of points $x_0,\dots,x_m\in F$ such that $x_0(i)=0$, $x_m(i)=k_i-1$, and $d(x_{j-1},x_j)\le 1$ for all $j\in\{1,\dots,m\}$.  

By the Hex Theorem of Gale \cite{Gale}, for any discrete box $K$ of dimension $n$ and any cover $\F$ of $K$ of cardinality $|\F|\le n$, some set $F\in\F$ connects two opposite faces of $K$. The following theorem can be seen as a local generalization of the Hex Theorem of Gale.

\begin{theorem}\label{t:Hex} For any cover $\F$ of a discrete box $K$ of dimension $n$, either some set  $F\in\F$ connects two opposite faces of $K$, or there exists a subset $B\subseteq K$ of diameter $\le 1$ that intersects at least $n+1$ elements of the cover $\F$.
\end{theorem}

Theorem~\ref{t:Hex} will be derived from the following topological result that can be found in  \cite[1.8.20]{Eng}.

\begin{lemma}[Lebesgue Covering Lemma]\label{lem} If $\mathcal F$ is a finite closed cover of the cube $[0,1]^n$, no member of which meets two opposite faces of $[0,1]^n$, then there is a point $x\in[0,1]^n$ that belongs to at least $n+1$ elements of the cover $\F$.
\end{lemma}

\begin{proof}[Proof of Theorem~\ref{t:Hex}] Let $\F$ be a cover of a discrete box $K=k_1\times \dots\times k_n$ of dimension $n$ such that for every subset $B\subseteq K$ of diameter $\le 1$ meets at most $n$ elements of the cover $\
F$. We should prove that some set $F\in\F$ connects two opposite faces of the box $K$.

For every $F\in \F$ and $x\in F$ let $F(x)$ be the set of all point $y\in F$ for which there exists a chain $x=x_0,x_1,\dots x_m=y$ in $F$ such that $d(x_{i-1},x_i)\le 1$ for all $i\in\{1,\dots,m\}$. Then $\U_F=\{F(x):x\in F\}$ is a partition of $F$ into chain-connected sets such that $$\dist(F(x),F(y))=\min\{d(u,v):u\in F(x),\;v\in F(y)\}\ge 2$$ for any distinct sets $F(x),F(y)\in\U_F$.

It follows that $\U=\bigcup_{F\in\F}\U_F$ is a cover of $K$. The discrete box $K=k_1\times\dots\times k_n$ can be considered as the ``integer part'' of the ``continuous'' box $\widetilde K=\prod_{i=1}^n[0,k_i-1]$. The box $\widetilde K$ is endowed with the metric $$\tilde d(x,y)=\max_{1\le i\le n}|x(i)-y(y)|.$$
Here we identify the elements of the box $\widetilde K$ with functions $x:\{1,\dots,n\}\to\IR$ such that $0\le x(i)\le k_i-1$ for all $i\in\{1,\dots,n\}$.   

For every subset $A\subseteq K$ consider its closed $\frac12$-neighborhood $$\widetilde A=\{x\in \widetilde K:\exists a\in A\;\;\tilde d(x,a)\le\tfrac12\}$$in $\widetilde K$. Taking into account that $\tilde d(a,b)\ge 1$ for any distinct points of $K$, we conclude that for two sets $A,B\subseteq K$ are equal if and only if $\widetilde A=\widetilde B$. 

It follows that $\widetilde \U=\{\widetilde U:U\in\U\}$ is a finite closed cover of $\widetilde K$. We claim that for every $x\in\widetilde K$ the family $\U_x=\{U\in\U:x\in \widetilde U\}$ has cardinality $\le n$. To derive a contradiction, assume that $|\U_x|>n$ for some $x\in\widetilde K$. For every $U\in\U_x$ choose a point $y_U\in U$ with $\tilde d(x,y_U)\le\frac12$ and find a set $F_U\in \F$ such that $U=F_U(y_U)$. Since the set $Y=\{y_U:U\in\U_x\}$ has diameter $\le 1$, our assumption guarantees that the family $\F'=\{F\in\F:F\cap Y\ne\emptyset\}$ has cardinality $|\F'|\le n$. Since $\{F_U:U\in\U_x\}\subseteq\F'$, by the Pigeonhole Principle, there exists a set $F\in\F'$ such that $F=F_U=F_{V}$ for some distinct sets $U,V\in\U_x$. Then $U=F(y_U)$ and $V=F(y_{V})$. Observing that
$$d(y_U,y_{V})\le \tilde d(y_U,x)+\tilde d(x,y_{V})\le\tfrac12+\tfrac12=1,$$
we conclude that $U=F(y_U)=F(y_{V})=V$, which contradicts the choice of the sets $U,V$. This contradiction shows that $|\U_x|\le n$.  Then by Lemma~\ref{lem}, there exists a set $U\in\U$ such that the set $\widetilde U$ meets two opposite faces of $\widetilde K$. Then the set $U$ meets two opposite faces of the box $K$ and being chain-connected, connects these two opposite faces of $K$. Choose any point $u\in U$ and find a set $F\in\F$ such that $U=F(u)$. Since $U=F(u)\subseteq F$, the set $F$ connects two opposite faces of the box $K$.
\end{proof}

\section{Two Lemmas related to Asymptotic dimension}\label{s:l}

In the proof of Theorems~\ref{t:main} and \ref{t:main2} we shall exploit two lemmas. The first of them characterizes coarse spaces of asymptotic dimension zero in terms of chains. For an entourage $E$ on a set $X$, by an {\em $E$-chain} we understand  a sequence of points $x_0,\dots,x_n$ in $X$ such that $(x_{i-1},x_i)\in E$ for every $i\in\{1,\dots,n\}$.

\begin{lemma}\label{zero} A coarse space $(X,\E)$ has asymptotic dimension zero if and only if for any entourage $E\in \E$ there exists an entourage $F\in\E$ such that $(x_0,x_n)\in F$ for any $E$-chain $x_0,\dots,x_n$ in $X$.
\end{lemma}

\begin{proof} To prove the ``only if'' part, assume that $\asdim (X,\E)=0$. Then for every $E\in\E$ there exists a uniformly bounded cover $\U$ of $X$ such that for every $x\in X$ the $E$-ball $E[x]$ meets a unique set $U\in\U$.
By the uniform boundedness of $\U$, there exists an entourage $F\in\E$ such that $U\times U\subseteq F$ for every $U\in\U$. Now take any $E$-chain $x_0,\dots,x_n\in X$. Find a unique set $U\in\U$ containing the point $x_0$. By induction we shall show that $x_i\in U$ for any $i\le n$. For $i=0$ this follows from the choice of $U$. Assume that for positive $i<n$ we have proved that $x_{i}$ belongs to $U$. By the choice of the cover $\U$, the $E$-ball $E[x_{i}]$ does not intersect any set of the family $\U\setminus\{U\}$ and hence $x_{i+1}\in E[x_i]\subseteq U$. Therefore $x_n\in U$ and $(x_0,x_n)\in U\times U\subseteq F$.
\smallskip

To prove the ``if'' part, assume that for any entourage $E\in \E$ there exists an entourage $F\in\E$ such that $(x_0,x_n)\in F$ for any $E$-chain $x_0,\dots,x_n$ in $X$. For every $x\in X$ let $U_x$ be the set of all points $y\in X$ for which there exists an $E$-chain $x=x_0,x_1,\dots,x_n=y$. The choice of the entourage $F$ guarantees that $U_x\subseteq F[x]$ and hence the cover $\U=\{U_x:x\in X\}$ of $X$ is uniformly bounded. It is easy to see that for each $x\in X$ the set $U_x$ is the unique set in $\U$ that meets the $E$-ball $E[x]$. Therefore, the cover $\U$ witnesses that $\asdim (X,\E)=0$.
\end{proof}
  
In the following lemma we extend the ``only if'' part of the characterization given in Lemma~\ref{zero}  to higher dimensions. 

Let $E$ be an entourage on a set $X$. An {\em $E$-box of dimension $n$} in $X$ is a map $f:K\to X$ defined on some discrete box $K$ of dimension $n$ such that $(f(x),f(y))\in E$ for any points $x,y\in K$ with $d(x,y)\le 1$.

\begin{lemma}\label{l:higher} If for some positive interger $n$, a coarse space $(X,\E)$ 
has $\asdim(X,\E)<n$, then for every entourage $E\in\E$ there exists an entourage $F\in\E$ such that for any $E$-box $f:K\to X$ of dimension $n$ there exists a set $V\subseteq K$ connecting two opposite faces of $K$ such that $f(V)\times f(V)\subseteq  F$.
\end{lemma}

\begin{proof} Assuming that $\asdim(X,\E)<n$, for any entourage $E\in\E$ we can find a uniformly bounded cover $\U$ of $X$ such that $|U\in\U:E[x]\cap U\ne\emptyset\}|\le n$ for every $x\in X$. By the uniform boundedness of $\U$, there exists an entourage $F\in\E$ such that $U\times U\subseteq F$ for every $U\in\U$. We claim that the entourage $F$ has the required property. Let $f:K\to X$ be any $E$-box of dimension $n$ in $X$. The cover $\U$ of $X$ induces the cover $\V=\{f^{-1}(U):U\in\U\}\setminus\{\emptyset\}$ of $K$. Consider the entourage $B=\{(x,y)\in K:d(x,y)\le 1\}$ on the discrete box $K$. The definition on an $E$-box implies that for every $x\in K$ we have $f(B[x])\subseteq E(f[x])$ and hence $\{U\in\U:f^{-1}(U)\cap B[x]\ne\emptyset\}\subseteq \{U\in \U:U\cap E[f(x)]\ne\emptyset\}$, which implies that $|\{V\in\V:V\cap B[x]\ne\emptyset\}|\le n$ for every $x\in K$. By Theorem~\ref{t:Hex}, some set $V\in\V$ connects two opposite faces of the box $K$.
It is clear that $f(V)\times f(V)\subseteq F$. 
\end{proof}

\section{Proof of Theorem~\ref{t:main}}\label{s:pf1}

Let $X_1,\dots,X_n$ be coarse spaces of asymptotic dimension $\ge 1$ and $X=X_1\times\dots\times X_n$ be their products. Let $\E$ be the course structure of the space $X$ and for every $i\in\{1,\dots,n\}$ let $\E_i$ be the coarse structure of the coarse space $X_i$. 

By Lemma~\ref{zero}, for every $i\in\{1,\dots,n\}$ there exists an entourage $E_i\in\E_i$ such that for every $D_i\in\E_i$ there exists an $E_i$-chain $x_0,\dots,x_n\in X$ such that $(x_0,x_n)\notin D_i$. The entourages $E_1,\dots,E_n$ induce the entourage $$E=\{((x_i)_{i=1}^n,(y_i)_{i=1}^n)\in X\times X:((x_i,y_i))_{i=1}^n\in \prod_{i=1}^n E_i\}\in\E.$$

Assuming that $\asdim(X)<n$ we can apply Lemma~\ref{l:higher} and find an entourage $F$ such that such that for any $E$-box $f:K\to X$ of dimension $n$ there exists a set $C\subseteq K$ connecting two opposite faces of $K$ such that $f(C)\times f(C)\subseteq F$. Replacing $F$ by a larger entourage, we can assume that
$$F=\{((x_i)_{i=1}^n,(y_i)_{i=1}^n)\in X\times X:((x_i,y_i))_{i=1}^n\in \prod_{i=1}^n F_i\}$$
for some entourages $F_1\in\E_1,\dots,F_n\in\E_n$.

For every $i\in\{1,\dots,n\}$ the choice of the entourage $E_i$ yields an $E_i$-chain $x_{0,i},\dots,x_{m_i,i}\in X_i$ such that  $(x_{0,i},x_{m_i,i})\notin  F_i$.

Consider the discrete box $K=(m_1+1)\times\dots\times (m_n+1)$, and the map 
$$f:K\to X,\;\;f:(k_1,\dots,k_n)\mapsto (x_{k_1,1},\dots,x_{k_n,n}),$$
which is an $E$-box of dimension $n$ in $X$. By the choice of the entourage $F$, 
there exists a set $C\subseteq K$ that connects two opposite faces of the box $K$ and has $f(C)\times f(C)\subseteq F$. Then for some $i\in\{1,\dots,n\}$, the projection of $C$ onto the $i$-th coordinate axis coincides with the set $m_i+1=\{0,\dots,m_i\}$ and the projection of $f(C)$ onto the $i$-coordinate coincides with the chain $x_{0,i},\dots,x_{m_i,i}$. Taking into account that $f(C)\times f(C)\subseteq F$, we conclude that $(x_{0,i},x_{m_i,i})\in F_i$, which contradicts the choice of the chain $x_{0,i},\dots,x_{m_i,i}$. This contradiction completes the proof of the inequality $\asdim(X)\ge n$.

\section{Proof of Theorem~\ref{t:main2}}\label{s:pf2}

Let $(X,\E)$ be a finitary coarse space admitting an almost free action of the group $\IZ^n$. Then $\IZ^n$ is isomorphic to a subgroup $G\subseteq S_X$ such that $\E_G\subseteq \E$ and for any finite set $F\subseteq G\setminus\{e\}$ there exists a point $x\in X$ such that the subgroup $\{g\in G:gx=x\}$ is disjoint with $F$. 

 Let $h:\IZ^n\to G$ be an isomorphism. Since $\E_G\subseteq\E$, the entourage $$E=\{(x,y)\in X\times X:y\in h(\{-1,0,1\}^n)\cdot x\}$$belongs to the coarse structure $\E$. Assuming that $\asdim (X,\E)<n$, we can apply Lemma~\ref{l:higher} and find an entourage $F\in\E$ such that any $E$-box $f:K\to X$ of dimension $n$ contains a subset $C\subseteq K$ that connects two opposite faces of $K$ and has $f(C)\times f(C)\subseteq F$.
 
Since the coarse structure $\E$ is finitary, the number $m=\sup_{x\in X}|F[x]|$ is finite. Since the action of the group $G$ is almost free, there exists a point $x\in X$ such that the subgroup $\{g\in G:gx=x\}$ of $G$ is disjoint with the finite set $h(\{0,\dots,m\}^n\setminus \{0\}^n)$. Then the map $f:\{0,\dots,m\}^n\to X$, $f:z\mapsto h(z)\cdot x$, is an injective $E$-box of dimension $n$ in $X$. By the choice of the entourage $F$, there exists a set $C\subseteq \{0,\dots,m\}^n$ that connects two opposite faces of the discrete box $\{0,\dots,m\}^n$ and has $f(C)\times f(C)\subseteq F$. Then $f(C)\subseteq F[f(c)]$ for any $c\in C$. The injectivity of the map $f$ implies that $|C|\le|F[f(c)]|\le m$ which is not possible as any set connecting two opposite faces of the box $\{0,\dots,m\}^n$ has cardinality $>m$. This contradiction shows that $\asdim(X,\E)\ge n$.

\section{Acknowledgement}

The authors express their sincere thanks to Igor Protasov for inspiring and stimulating questions.


\end{document}